\documentclass[12pt,leqno,oneside]{amsart}
\usepackage{amsmath,amssymb,amsthm}

\newcommand{\N}{{\mathbb N}}

\newcommand{\R}{{\mathbb R}}
\newcommand{\C}{{\mathbb C}}
\newcommand{\F}{\mathbb{F}}
\newcommand{\I}{\mathbb{O}}
\newcommand{\M}{\mathbb{M}}

\newcommand{\ds}{\displaystyle}

\newcommand{\avoids}{\;\#\;}
\newcommand{\ep}{\varepsilon}

\newcommand{\be}{\begin{equation}}
\newcommand{\ee}{\end{equation}}

\newtheorem*{thm*}{Theorem}

\newtheorem*{lem*}{Lemma}
\newtheorem{prop}{Proposition}
\newtheorem*{prop*}{Proposition}

\newtheorem*{cor*}{Corollary}

\theoremstyle{definition}

\newtheorem*{defn*}{Definition}

\newtheorem*{example*}{Example}

\newtheorem*{examples*}{Examples}

\newtheorem*{rmk*}{Remark}

\numberwithin{equation}{section}
\allowdisplaybreaks

\begin{document}
\title{The ring of fluxions} 
\author{S.R.~Doty} \address{Mathematics and Statistics, Loyola
  University Chicago, Chicago, Illinois 60626 U.S.A.}
\email{doty@math.luc.edu}

\begin{abstract}\noindent
The ring of fluxions (real sequential germs at infinity) provides a
rigorous approach to infinitesimals, different from the better-known
approach of Abraham Robinson. The basic idea was first espoused in a
paper by Curt Schmieden and Detlof Laugwitz published in
1958. Although this ring codifies all the usual intuitive properties
of infinitesimals in a very elementary way, its existence has been
generally ignored.
\end{abstract}

\maketitle

\section*{Introduction}\noindent
In 1960 Abraham Robinson \cite{R1, R2} discovered that a rigorous
theory of infinitesimal real numbers can be developed in terms of the
existence of nonstandard models of the real number field. Robinson's
approach, which depends on the compactness theorem of first-order
logic, is not easily accessible to the mathematical community at
large.

It is not widely known that two years prior to Robinson's insight,
C.~Schmieden and D.~Laugwitz published a paper \cite{SL} in which a
more elementary approach to infinitesimals was outlined. This approach
is based on the construction of an ordered ring extension $\dot{\R}$
of the real ordered field $\R$, consisting of germs of real sequences
at infinity.  In this article, I have taken the liberty to introduce
the term \emph{fluxions} for such germs, adapting an obsolete term
from \cite{Newton}.

According to \cite[p.~2]{R2}, ``G.W.~Leibnitz argued that the theory
of infinitesimals implies the introduction of ideal numbers which
might be infinitely small or infinitely large compared with real
numbers but which were \emph{to possess the same properties as the
  latter}.'' This of course demands not just an ordered ring extension
but an ordered field extension, which is what Robinson's method
produces. Yet the earlier method of Schmieden and Laugwitz seems to me
more generally accessible and intuitively appealing, and the
difficulties imposed by having a ring instead of a field in which to
house the infinitely small and large quantities are minor.

Although Cauchy is generally credited with the modern epsilon-delta
definition of limit, a look at his 1821 \emph{Cours d'analyse} reveals
no such definition stated explicitly. Instead we find the following
suggestive statements \cite{Cauchy}:
\begin{quote}\small\sf
We call a quantity variable if it can be considered as able to
take on successively many different values.

When the values successively attributed to a particular variable
indefinitely approach a fixed value in such a way as to end up by
differing from it by as little as we wish, this fixed value is called
the limit of all the other values.

When the successive numerical values of the same variable decrease
indefinitely, such as to become less than a given number, this number
becomes what is called an infinitesimal or an infinitesimal
quantity. The limit of this type of variable is zero.
\end{quote}
This language seems consistent with the approach to convergence via
the ring of fluxions, and suggests \cite{L87, L89} that the idea may
have been anticipated by Cauchy. Intuitively, a fluxion is merely a
sequence of real numbers, with an equality rule that says two
sequences define the same fluxion whenever they differ in only a
finite number of terms. A fluxion (\emph{e.g.}, $[1/n]$) is
infinitesimal when it eventually becomes smaller in magnitude than any
given positive real number. Inequalities and real-valued functions of
a real variable extend easily to fluxions.

The purpose of this article is to demonstrate how a number of standard
results of basic analysis could be formulated and proved using the
ring of fluxions. It is far from comprehensive; indeed I was forced to
be rather selective of topics in order to keep the article to a
reasonable length. Once the ring of fluxions has been precisely
defined, and its ordering properties developed, the task of rewriting
standard definitions and arguments in the fluxional language is in
most cases quite straightforward, with no loss of mathematical
rigor. In some instances (\emph{e.g.}  uniform continuity in Section
\ref{sec:continuity}) concepts may seem more intuitively appealing
when expressed in the new language.

\section{The ordered ring $\dot{\R}$ of fluxions}\noindent
Taking the ordered field $\R$ of ordinary real numbers as given
\cite{Landau}, along with its least upper bound property, we wish to
construct an ordered ring extension $\dot{\R}$ of $\R$ which contains
infinitely small and infinitely large quantities. Elements of
$\dot{\R}$ will be called fluxions.

In order to construct the ring $\dot{\R}$, consider first the larger
structure $\R^\N$ of all sequences $x \colon \N \to \R$ of real
numbers. Here, $\N$ is the set of natural numbers; \emph{i.e.}, positive
integers.  As usual, given a sequence $x \colon \N \to \R$ in $\R^\N$,
we also write $(x_n)$ or $(x_n)_{n \in \N}$ or $(x_n)_{n=1}^\infty$
for it.

The set $\R^\N$ is a commutative ring with the usual ring operations
of componentwise addition and componentwise multiplication of
sequences:
\begin{equation} 
  (x_n)+(y_n) = (x_n+y_n); \qquad (x_n)\cdot(y_n) = (x_ny_n).
\end{equation} 
Let $Z$ be the set of sequences $(x_n)$ which are \emph{eventually
  zero}, meaning that $x_n$ is non-zero for only finitely many values
of $n$. In other words, $(x_n) \in Z$ if and only if there is some
natural number $N$ for which $x_n=0$ for all $n \ge N$. It is easily
checked that $Z$ is an ideal in the ring $\R^\N$. We define $\dot{\R}$
to be the corresponding quotient ring.

\begin{defn*}
  $\dot{\R} := (\R^\N/Z)$, as a quotient ring.
\end{defn*}

It is useful to introduce a binary relation on the set $\R^\N$,
denoted by the symbol $\doteq$, by declaring that
\begin{equation*}
  (x_n) \doteq (x'_n) \iff x_n = x'_n \text{ for all but finitely many
    values of } n.
\end{equation*}
Equivalently, $(x_n) \doteq (x'_n)$ if and only if there is some
natural number $N$ such that $x_n=x'_n$ for all $n \ge N$. Thus, the
relation $\doteq$ is the relation of \emph{eventual}, or \emph{almost
  everywhere}, equality. One may easily check that $\doteq$ is an
equivalence relation on $\R^\N$. Thus we say that the sequences
$(x_n)$ and $(x'_n)$ are \emph{equivalent} if and only if $(x_n)
\doteq (x'_n)$. Observe that $(x_n) \doteq (x'_n)$ if and only if
$(x_n-x'_n) \in Z$. So cosets in the quotient ring $\R^\N/Z$ are
precisely the same thing as the equivalence classes under
$\doteq$. Thus, $\dot{\R}$ is the set of equivalence classes of
real-valued sequences under eventual equality. Intuitively, then, a
fluxion may be regarded as a real sequence, with two sequences which
are eventually equal representing the same fluxion.

We denote the image of a sequence $(x_n)$ under the natural quotient
map $\R^\N \to \dot{\R}$ by $[x_n]$. The alternative notations
$[x_n]_{n \in \N}$, $[x_n]_{n=1}^\infty$ could also be used.  Addition
and multiplication of fluxions is therefore defined by
\begin{equation} 
  [x_n]+[y_n] = [x_n+y_n]; \qquad [x_n]\cdot [y_n] = [x_ny_n].
\end{equation} 
There is a natural embedding of $\R$ into $\dot{\R}$, given by sending
a given real number $a$ to the corresponding constant fluxion
$[a]:=[a,a,a,a,\dots]$. Under the embedding, the additive and
multiplicative identities in $\R$ (0 and 1) correspond to the additive
and multiplicative identities in $\dot{\R}$. It is clear that this
embedding is a ring homomorphism; \emph{i.e.}, addition and multiplication of
real numbers corresponds to addition and multiplication of the
corresponding fluxions. Henceforth we will regard ordinary real
numbers as fluxions in this way, identifying them with their
associated constant fluxion.

It is important to notice that the ring $\dot{\R}$ is not a
field. Indeed, the fluxion $[(-1^n)+1] = [0,2,0,2,\dots]$ is a zero
divisor, so $\dot{\R}$ is not even an integral domain.

From now on we will often use single symbols such as $x, y, z$ to
denote fluxions (elements of $\dot{\R}$).  The usual ordering relation
$\le$ on real numbers extends to a (partial) ordering of $\dot{\R}$,
which will be denoted by the same symbol.

\begin{defn*}
  Given fluxions $x=[x_n]$, $y=[y_n]$, define $x \le y$ if and only if
  $x_n \le y_n$ for all but finitely many $n$. Whenever $x \le y$ we
  say that $x$ is \emph{eventually} less than or equal to $y$. We also
  write $x \ge y$ if and only if $y \le x$.  Whenever $x \ge y$ we say
  that $x$ is \emph{eventually} greater than or equal to $y$.
\end{defn*}

Clearly, given real numbers $a,b$ we have $[a] \le [b]$ if and only if
$a \le b$. So, when restricted to pairs of real numbers, $\le$ has the
usual meaning. It is easily checked that $\le$ is a partial order on
$\dot{\R}$. It is not a linear order, however. For example, if $x =
[0,2,0,2,\dots] = [1+ (-1)^n]$ and $y = [2,0,2,0, \dots] = [1 -
  (-1)^n]$ then $x$ and $y$ are incomparable: neither $x \le y$ nor $y
\le x$ holds. 

\begin{defn*}
  Given fluxions $x=[x_n]$, $y=[y_n]$, we write $x < y$ if and only if
  $x_n < y_n$ for all but finitely many $n$. Whenever $x < y$ we say
  that $x$ is \emph{eventually} less than $y$.  We also write $x > y$
  if and only if $y < x$.  Whenever $x > y$ we say that $x$ is
  \emph{eventually} greater than $y$.
\end{defn*}

Obviously $x < y$ implies $x \le y$ and similarly $x > y$ implies $x
\ge y$.  It should immediately be pointed out that in the realm of
fluxions, $x \le y$ and $x \ne y$ do not necessarily imply that $x <
y$. For example, if $x = [0,2,0,2, \dots] = [1+(-1)^n]$ then $x \le 2$
and $x \ne 2$, yet $x \not< 2$ (\emph{i.e.}, $x$ isn't eventually
smaller than 2). Similarly for $\ge$. Thus, a certain amount of care
is required when using the ordering relations of fluxions. The main
properties of $\le$ and $<$ are as follows:
\begin{gather}
 \text{$x \le x$ \quad(reflexivity of $\le$);}\\
 \text{$x \le y$ and $y \le x$ imply $x=y$ \quad(antisymmetry of
$\le$);}\\
\text{$x\le y$ and $y \le z$ imply $x \le z$
\quad(transitivity of $\le$);}\\
\text{$x< y$ and $y < z$ imply $x < z$ \quad(transitivity
of $<$);}\\
\text{$x \le y \iff x-y \le 0$, and $x < y \iff x-y < 0$.}
\end{gather}
Similar statements hold when $\le$ and $<$ are
replaced by $\ge$ and $>$.

\begin{defn*}[Absolute value]
  Given a fluxion $x = [x_n]$, define its magnitude $|x|$ to be the
  fluxion $[\,|x_n|\,]$.
\end{defn*}

Notice that this extends the usual notion of absolute value on real
numbers.  The basic properties of absolute value are essentially the
same for $\dot{\R}$ as for $\R$. More precisely, we have the
following.

\begin{prop}
For any fluxions $x,y \in \dot{\R}$ we have:

\item{\upshape(a)}\quad $|x-y| \ge 0$ with equality if and only if
  $x=y$;
 
\item{\upshape(b)}\quad $|x-y| = |y-x|$; 

\item{\upshape(c)}\quad $|x+y| \le |x|+|y|$ (triangle inequality);

\item{\upshape(d)}\quad If $y \ge 0$ then $|x| \le y$ if and only if
  $-y \le x \le y$, and similarly $|x| < y$ if and only if $-y < x <
  y$.
\end{prop}

The proof is immediate from the definitions.

For $x=[x_n]$, $y=[y_n] \in \dot{\R}$, to say that $x \ne y$ means that
$x_n \ne y_n$ for infinitely many values of $n$. There is a stronger
form of this relation, in which inequality eventually holds
everywhere, formulated as follows.
\begin{defn*}
  Let $x=[x_n], y=[y_n] \in \dot{\R}$. We say that $x$ \emph{avoids}
  $y$ (written as $x \avoids y$) if there is some $N$ such that $x_n \ne
  y_n$ for all $n \ge N$.
\end{defn*}

Obviously, if $x\avoids y$ then $x \ne y$, but not conversely. For
example, take $x = 0$ and let $y = [1+(-1)^n]$. Then $x \ne y$ but it
isn't true that $x \avoids y$.

The reader should verify that if $x > 0$ or $x < 0$ then $x \avoids
0$; \emph{i.e.}, if $x$ is eventually positive or eventually negative, then
$x$ avoids $0$.  Fluxions that avoid 0 are important because they are
precisely the invertible elements (units) of the ring $\dot{\R}$.

\begin{prop}
  Let $x, y \in \dot{\R}$. Then $x$ avoids $y$ if and only if $x-y$ is
  invertible in $\dot{\R}$. In particular, $x \in \dot{\R}$ in
  invertible if and only if $x$ avoids $0$.
\end{prop}

The proof is easy.

Since $\dot{\R}$ is commutative, it makes sense to define the fraction
$\frac{x}{y}$ by $\frac{x}{y} := xy^{-1}$ whenever $y$ is invertible.
Thus division in $\dot{\R}$ is well-defined, provided only that the
denominator is invertible (\emph{i.e.}, avoids 0).

To summarize: fluxions may be regarded as real sequences under the
equivalence of eventual equality (similar to the way rational numbers
are regarded as fractions under an equivalence). Fluxions are added,
subtracted, multiplied, and divided componentwise, and this extends
the usual addition, subtraction, multiplication, and division on real
numbers, which are embedded in $\dot{\R}$ as the constant sequences.
Division makes sense only when the denominator is eventually nonzero
(avoids zero).  Inequalities also extend componentwise, but the
notions make sense only provided that they hold eventually.  Thus the
ring of fluxions merely formalizes the standard operations on
sequences, familiar from basic analysis.

\section{Infinitesimals}\noindent
We are ready to define infinitesimals and apply them to discuss
convergence. This formalizes the customary intuition that most people
develop after achieving an initial understanding of convergence.

In order to foster readability, from now on we will typically write
fluxions as single letters towards the end of the alphabet ($x$, $y$,
$z$, \dots), reserving symbols at the beginning of the alphabet ($a$,
$b$, $c$, \dots) to depict ordinary real numbers. 

\begin{defn*}
  An \emph{infinitesimal} is a fluxion $x$ such that $|x| < a$ for
  every positive real number $a$.  A fluxion $x$ is \emph{finite} (or
  \emph{bounded}) if there exists some real $a>0$ such that $|x| < a$.
  We write $x \to a$ (and say that $x$ \emph{approaches} $a$ or $x$
  \emph{converges} to $a$), for a real number $a$, if the difference
  $x-a$ is infinitesimal. Finally, we say that $x$ \emph{diverges} if
  $x$ does not converge to any real number.
\end{defn*}

When $x \to a$ we also call the real number $a$ the \emph{limit} of
$x$, so we have already defined limits of sequences. One may prefer to
employ one of the more cumbersome notations
\[
\text{$\lim x = a$, $\lim x_n = a$, or $\ds\lim_{n\to \infty} x_n = a$}
\]
as an alternative notation for $x \to a$.

\begin{examples*}
As an example of a nonzero infinitesimal, consider $h = [1/n] \in
\dot{\R}$.  Notice that $x$ is infinitesimal if and only if $|x|$ is
infinitesimal if and only if $-x$ is infinitesimal.  Clearly $0$ is
infinitesimal; moreover, $0$ is the only real infinitesimal.  It is
clear that any real number or infinitesimal is a finite
fluxion. Furthermore, any convergent fluxion is finite, because $x -a
= h$ with $h$ infinitesimal implies that $|x| = |a+h| \le |a|+|h| <
|a|+1$.
\end{examples*}

Algebraic interactions among finite fluxions are encapsulated in the
following result.

\begin{prop}
  The set $\F$ of finite fluxions is a subring of $\dot{\R}$, and the
  set $\I$ of infinitesimals is an ideal in $\F$. The set $\M$ of
  convergent fluxions is a subring of $\F$, and $\M = \R \oplus \I$.
\end{prop}

\begin{proof}
Let $x, y \in \F$. Then there are positive reals $a,b$ such that
$|x|<a$ and $|y|<b$. Hence $|x-y| \le |x| + |y| < a+b$, so $x-y \in
\F$. Since $|xy| = |x| |y| < ab$, it also follows that $xy \in \F$. So
$\F$ is a subring of $\dot{\R}$.

Now let $h, i \in \I$ and $x \in \F$. So $|x|<a$ for some positive
real $a$, while $|h| < \ep$ and $|i|<\ep$ for every positive real
$\ep$. Then $$|h-i| \le |h|+|i| < \ep+\ep = 2\ep,$$ for every positive
real $\ep$, so $h-i \in \I$. Moreover, $|xh| = |x| |h| < a \ep$, for
every positive real $\ep$, so $xh \in \I$.

Finally, $\R \cap \I = \{0\}$ since $0$ is the only real
infinitesimal. Any convergent $x$ is finite, and can be written in the
form $x = a+h$ for some real $a$, and some infinitesimal $h$.
\end{proof}

The proposition says, in particular, that \emph{sums, differences, and
  products of finite fluxions are finite}, and similarly \emph{sums,
  differences, and products of infinitesimals are
  infinitesimal}. Furthermore, \emph{a product of a finite fluxion and
  an infinitesimal is infinitesimal}. Finally, \emph{every convergent
  fluxion is uniquely expressible as the sum of a real number and an
  infinitesimal.}

Notice that it follows immediately from the proposition that limits
are unique: if $x \to a$ and $x \to a'$ for $a,a' \in \R$ then $a=a'$.
For $x-a=h$ and $x-a'=h'$ where $h,h'$ are infinitesimal, so the
difference $a-a'=h'-h$ is infinitesimal, but since both $a,a'$ are
real this forces $a=a'$, as desired.

It does not seem possible to obtain a meaningful geometric model of
the entire ring $\dot{\R}$ of fluxions. However, if we restrict our
attention to the subring $\M$ of convergent fluxions, then we could
define the \emph{monad} of a real number $a$ to be $\text{monad}(a) :=
\{ x \in \dot{\R} \colon x \to a \}$, and visualize this set as an
infinitesimal ``cloud'' surrounding $a$, intersecting the real line in
just the point $a$. While $\M$ itself is not linearly ordered, the
monads are: $\M$ is the disjoint union of its monads, and could be
roughly visualized as a ``cloudy'' line in which each point on the
usual real line has been expanded to its corresponding monad.

Next we consider infinite limits. 

\begin{defn*}
Let $x$ be a fluxion. Write $x \to \infty$ ($x$ \emph{approaches
  infinity}) if $x > a$ for every real $a$. Similarly, write $x \to
-\infty$ ($x$ \emph{approaches negative infinity}) if $x < a$ for
every real $a$.
\end{defn*}

For example, $[n] \to \infty$ and $[-n] \to -\infty$. We should
reiterate that if $x \in \dot{\R}$ approaches $\pm \infty$ then $x$
diverges. Moreover, there are many examples (\emph{e.g.}, $[(-1)^n n]$,
$[(-1)^n]$) of fluxions which diverge, yet neither approach infinity
nor approach minus infinity.

The following is a useful characterization of fluxions which diverge
to $\pm \infty$. Let us introduce the notations $\I^+ := \{x\in
\I\colon x>0\}$, $\I^- := \{x\in \I\colon x<0\}$. These are the sets
of eventually positive or eventually negative infinitesimals.

\begin{prop}
  Let $x$ be a fluxion. Then $x \to \infty \iff 1/x \in \I^+$ and $x
  \to -\infty \iff 1/x \in \I^-$.
\end{prop}

\begin{proof}
Suppose $x \to \infty$. Then $x > a$ for any real number $a$, so
$x>0$. Thus $x$ is invertible, $1/x > 0$, and since $1/x < 1/a$ for
any real $a>0$, it follows that $1/x$ is smaller than any positive
real number, hence is infinitesimal.  Conversely, if $1/x>0$ is
infinitesimal then $0 < 1/x < a$ for every positive real number $a$,
so $x > 1/a$ for every positive real $a$, and it follows that $x > b$
for every $b \in \R$, so $x \to \infty$.

The negative case is similar.
\end{proof}

In particular, if $x \to \pm \infty$ then $x$ is invertible.  The
following result gives conditions under which the reciprocal of a
fluxion exists and is finite. This is often useful in calculations.

\begin{prop}\label{thm:invert}
Let $x$ be a fluxion.

\item{\upshape(a)} Suppose that $x$ is bounded away from $0$, meaning
  that there is some real number $b$ such that $0<b<|x|$. Then $x$
  avoids $0$, so $x^{-1}=1/x$ exists in $\dot{\R}$, and moreover $1/x$
  is finite.

\item{\upshape(b)} If $x \to a$ for $a \ne 0$ real, then
  $|a|-\frac{1}{n} < |x| < |a|+\frac{1}{n}$ for any natural number
  $n$. Hence $1/x$ exists and is finite.
\end{prop}

\begin{proof}
(a) Write $x = [x_n]$. There must be some $N$ for which $b< |x_n|$ for
  all $n \ge N$, so in particular (by replacing the first $N-1$ terms
  by something nonzero if necessary) we may assume that $x_n \ne 0$
  for all $n$, and thus $1/x = [1/x_n]$ makes sense. Clearly $b<|x|$
  implies $|1/x| < 1/b$, so $1/x$ is finite.

(b) Put $h=x-a$. Then $h$ is infinitesimal, so $|h|<\frac{1}{n}$ for
  any $n \in \N$. Then $|x| = |a+h| \le |a|+|h| < |a|+\frac{1}{n}$,
  for any $n$. Similarly, $|x|=|a-(-h)|\ge \big| |a|-|h| \big| > |a| -
  \frac{1}{n}$, for any $n$. In particular, by taking $n$ so large
  that $\frac{1}{n} < |a|$ (such an $n$ exists by the archimedean
  property of $\R$) we see that $|x| > |a|-\frac{1}{n}$, i.e, $x$ is
  bounded away from 0.
\end{proof}

The following summarizes how convergence interacts with algebraic
operations.

\begin{prop}[Algebra of Convergence] \label{AlgebraOfConvergenceThm}
  Suppose that $x \to a$ and $y \to b$ where $a,b$ are real. Let $c
  \in \R$. Then $x+y \to a+b$, $x-y \to a-b$, $cy \to cb$, $xy \to
  ab$, and $x/y \to a/b$ (provided $b \ne 0$).
\end{prop}

\begin{proof}
Let $x-a=h$, $y-b=k$ where $h$, $k$ are infinitesimals. Then
$(x+y)-(a+b) = h+k$, $(x-y)-(a-b) = h-k$, $cy-cb = ck$, $xy-ab =
xy-ay+ay-ab = (x-a)y+a(y-b) = hy+ak$, and $$\frac{x}{y} - \frac{a}{b} =
\frac{bx-ay}{by} = \frac{b(x-a)+a(b-y)}{by} = \frac{bh-ak}{by}$$ are
all infinitesimal.  (Proposition \ref{thm:invert} shows that $y$ is
invertible in $\dot{\R}$ if $b\ne 0$.)
\end{proof}

Next we consider how convergence interacts with inequalities.

\begin{prop}[Convergence and Inequalities]\label{le}
  Let $a,b$ be real numbers, and $x,y,z$ fluxions.

  \item{\upshape(a)} If $a+h \ge 0$ for some infinitesimal $h$, then $a
  \ge 0$.

  \item{\upshape(b)} If $x \le y$, $x \to a$ and $y \to b$ then $a \le
  b$.

  \item{\upshape(c)} If $x \le y \le z$, $x \to a$ and $z \to a$ then
  $y \to a$.
\end{prop}

\begin{proof}
(a) Assume $a < 0$. Then $-a > 0$ and $h \ge -a$. This contradicts the
  assumption that $h$ is infinitesimal.

(b) Write $x = a+h$, $y = b+k$ where $h,k$ are infinitesimals. Then
  $x-y = (a-b) + (h-k)$, so $(a-b) + (h-k) \ge 0$. By part (a), it
  follows that $a-b \ge 0$.

(c) Since $x-a \le y-a \le z-a$ it follows that $y-a$ is
  infinitesimal. (To be precise, $x-a \le y-a \le z-a$ implies that
  $|y_n-a| \le \max(|x_n-a|, |z_n-a|)$ for all sufficiently large $n$,
  which shows that $y-a$ is infinitesimal.) Thus $y \to a$. 
\end{proof}

The reader can easily verify that part (b) of the preceding result
admits the following generalization: if $x \le y$, $x \to a$ and $y
\to b$ for $a,b \in \R \cup \{\pm \infty\}$ then $a \le b$. For this
interpretation, one has to stipulate that $-\infty < a < \infty$ for
all $a \in \R$. It should also be noted that part (b) fails for strict
inequalities: it is possible to find examples where $x < y$, $x\to a$,
and $y\to b$ yet $a\not < b$.

The following result is a fundamental property of convergence. We say
that $x \in \dot{\R}$ is \emph{monotonic} if $x = [x_n]$ is eventually
either monotonically increasing ($x_n \le x_{n+1}$ for all
sufficiently large $n$) or monotonically decreasing ($x_n \ge x_{n+1}$
for all sufficiently large $n$).

\begin{thm*}[Monotone Convergence Theorem]
  A monotonic fluxion is finite if and only if it converges.
\end{thm*}

\begin{proof}
We argue only the increasing case, as the decreasing case is similar.
Assume $x$ is monotonically increasing, \emph{i.e.}, $x_n \le x_{n+1}$ for
all $n \ge N$. If $x$ converges then $x$ is finite. Conversely,
suppose that $x$ is finite. Let $E = \{ x_n \colon n \ge N\}$. The set
$E$ is a bounded set of real numbers, so its supremum $a = \sup E$
exists in $\R$. Then it is easy to see from the definition of supremum
that $x-a$ is infinitesimal, and hence $x \to a$.
\end{proof}

It should be clear that if a fluxion $x$ is monotonically increasing
yet not finite, then $x \to \infty$, and similarly, if $x$ is
monotonically decreasing yet not finite, then $x \to -\infty$.

\begin{defn*}
If $(x_{n_k})_{k \in \N}$ is a subsequence of a sequence $(x_n)_{n\in
  \N}$ defining a fluxion $x = [x_n]$, put $y = [x_{n_k}]$ and say
that the fluxion $y$ is a \emph{subfluxion} of $x$ (written as $y
\subset x$).
\end{defn*}

The Bolzano--Weierstrass theorem is another fundamental result in
analysis. The proof given below is taken from a Wikipedia article.

\begin{thm*}[Bolzano--Weierstrass Theorem]
  Any finite fluxion must have a convergent subfluxion.
\end{thm*}

\begin{proof}
Let $(x_n)$ be a sequence such that the corresponding fluxion
$x=[x_n]$ is finite. We claim that there must be a monotonic
subsequence of $(x_n)$; the corresponding fluxion must converge by the
monotone convergence theorem.  

To prove the claim, let us call a positive integer $n$ a \emph{peak}
of the sequence $(x_n)$ if $x_n > x_m$ for all $m > n$;  \emph{i.e.}, if 
$x_n$ is greater than every subsequent term in the sequence. Suppose
first that the sequence has infinitely many peaks, $n_1 < n_2 < n_3 <
\cdots < n_k < n_{k+1} < \cdots$. Then the subsequence $(x_{n_k})$ 
corresponding to peaks is monotonically decreasing, and we are
done. So suppose now that there are only finitely many peaks, let $N$
be the last peak and put $n_1 = N + 1$. Then $n_1$ is not a peak,
since $n_1 > N$, which implies the existence of an $n_2 > n_1$ with
$x_{n_2} \geq x_{n_1}$.  Again, any $n_2 > N$ is not a peak, hence
there is some $n_3 > n_2$ with $x_{n_3} \geq x_{n_2}$.  Repeating this
process leads to an infinite monotonically increasing subsequence
$x_{n_1} \leq x_{n_2} \leq x_{n_3} \leq \cdots$. Finally, if there are
no peaks at all, then for every element of the sequence, there must be
a subsequent larger element, which, in turn has a subsequent larger
element and so on, and these constitute a monotonic subsequence. This
proves the claim, and thus the result.
\end{proof}

We now discuss limits inferior and superior. For this we need to work
within the extended real-number system $\R \cup \{\pm \infty\}$. 

\begin{defn*}
  Suppose $x$ is a fluxion.  Let $E \subset \R \cup\{\pm \infty\}$ be
  the collection of all extended real numbers $a \in \R \cup\{\pm
  \infty\}$ for which there is some subfluxion $y \subset x$ such that
  $y\to a$. Define
  \[
     \limsup x = \sup E, \qquad  \liminf x = \inf E.
  \]
  The extended real numbers $\limsup x$ and $\liminf x$ are
  respectively called the \emph{limit superior} and \emph{limit
    inferior} of the fluxion $x$.
\end{defn*}

It should be clear to the reader that if $x \to a$ then also $y \to
a$, for any $y \subset x$. This is true for any $a \in \R \cup \{\pm
\infty\}$.  This leads to the following characterization of
convergence (and divergence to $\pm \infty$) in terms of the
limit superior and limit inferior.

\begin{prop}
 Let $x$ be a fluxion. Then $x \to a \in \R \cup \{\pm \infty\}$ if
 and only if $\limsup x = a = \liminf x$.
\end{prop}

\begin{proof} 
  Suppose $x \to a$. Then $y \to a$ for any $y \subset x$, and thus
  $E= \{a \in \R \cup\{\pm \infty\} \colon y \to a, y \subset x \}$
  is a singleton $\{a\}$, whence $\sup E = \inf E = a$. On the other
  hand, if $\limsup x = a = \liminf x$ then the set $E$ defined above
  must be a singleton $\{a\}$, and so every subfluxion of $x$ must
  approach $a$. In particular, $x \to a$.
\end{proof}

Here is an elementary application of the last result, which
illustrates the utility of these concepts. 

\begin{prop}\label{thm:nth-root}
  Let $x \ge 0$ be a fluxion and suppose that $x \to a$, where $a$ is
  real. Then $\sqrt[n]{x} \to \sqrt[n]{a}$, for any $n \in \N$.
\end{prop}

\begin{proof} (Based on a suggestion by R.~Jensen.)
First we claim that if a fluxion $y\ge 0$ diverges then $y^n$
diverges. (The claim is false without the assumption that $y \ge 0$.)
If $y$ isn't finite then neither is $y^n$, so the claim is proved in
that case. On the other hand, if $y$ is finite then by the preceding
result we must have $\limsup y \ne \liminf y$, so there exist
subfluxions $u,v \subset y$ with $u \to b$, $v \to c$ for real $b \ne
c$. Then $u^n, v^n \subset y^n$ and $u^n \to b^n$, $v^n \to c^n$, and
$b^n \ne c^n$ so $y^n$ diverges in this case, too.

Now put $y = \sqrt[n]{x}$. The contrapositive of the claim proved in
the preceding paragraph shows that $y$ converges. Let $d$ be the
limit, so $y \to d$. Then by Proposition \ref{AlgebraOfConvergenceThm}
it follows that $y^n \to d^n$, so by uniqueness of limits $d^n = a$
and thus $d = \sqrt[n]{a}$, as desired.
\end{proof}

Convergence of an infinite series $\sum_{n=1}^\infty a_n$ may be
defined as usual, in terms of convergence of the fluxion $s = [s_n] =
[a_1+\cdots + a_n]$.

\section{Continuity and differentiability}\label{sec:continuity}\noindent
We now consider continuity. For this, we need to extend real-valued
functions of a real variable to assume values at fluxions.

If $f$ is a real-valued function of a real variable, let $D \subset
\R$ be its domain. So $f$ maps $D$ into $\R$. Let $\dot{D}$ be the set
of all fluxions $x=[x_n]$ which are eventually in $D$, meaning that
$x_n \in D$ for all but finitely many $n$. Then $f$ extends to a
function, also denoted by $f$, from $\dot{D}$ into $\dot{R}$, by
defining $f(x)$ to be $[f(x_n)]$ if $x = [x_n] \in \dot{D}$.

\begin{defn*}
Let $f\colon D \to \R$ be a real-valued function of a real variable,
and let $a$ be a point of its domain $D$. We say that $f$ is
\emph{continuous at} $a$ if $f(x) \to f(a)$, for every fluxion $x \in
\dot{D}$ for which $x \to a$. One also says that $f$ is
\emph{continuous on} $D$ if $f$ is continuous at $a$ for every $a \in
D$.
\end{defn*}

Note that if $a$ is an isolated point of $D$, meaning that the only $x
\in \dot{D}$ converging to $a$ is $a$ itself, then any $f \colon D \to
\R$ is continuous at $a$. The definition also includes the notions of
left or right continuity, in case the point $a$ is a left or right
endpoint of $D$.

\begin{prop}[Algebra of Continuity]\label{AlgebraOfContinuity}
  Let $f$, $g$ be real-valued functions of a real variable, both
  continuous at some real point $a$. Then $f+g$, $f-g$, and $fg$ are
  continuous at $a$. Moreover, $f/g$ is continuous at $a$ provided
  that $g(a) \ne 0$.
\end{prop}

\begin{proof}
Combine the definition with Proposition \ref{AlgebraOfConvergenceThm}.
\end{proof}

As an application, notice that Proposition \ref{thm:nth-root} says
that the real-valued function $x \mapsto x^{1/n}$ is continuous on the
interval $[0, \infty)$. It then follows from the proposition that the
  real-valued function $x \mapsto x^{m/n}$ is continuous on the
  interval $[0, \infty)$; that is, taking rational powers is a
    continuous operation.

\begin{thm*}[Intermediate Value Theorem] 
  Let $f$ be a real-valued function defined and continuous (at least)
  on the interval $[a,b]$, such that $f(a) \ne f(b)$. If $i$ is any
  real number strictly between $f(a)$ and $f(b)$, then there exists
  some real number $c$ strictly between $a$ and $b$ such that $f(c) =
  i$.
\end{thm*}

\begin{proof}
A standard argument, based on the bisection method, easily translates
into the language of fluxions. First we observe that by replacing $f$
by $f-i$ we are reduced to the case where the intermediate value $i$
is zero, and $f(a)$ and $f(b)$ have different signs (i.e.,
$f(a)f(b)<0$). Then we put $x_0 = a$, $y_0= b$. Computing the midpoint
$m_0 = \frac{1}{2}(x_0+y_0)$ of the interval $[x_0, y_0]$, we have
three cases:
\begin{enumerate}
\item If $f(m_0) = 0$ then we are done (put $c = m_0$). 
\item Otherwise, if $f(m_0)$ has the same sign as $f(x_0)$, put $x_1 =
  m_0$ and $y_1 = y_0$.
\item Otherwise, put $x_1 = x_0$ and $y_1=m_0$. 
\end{enumerate}
Assuming we didn't already find $c$, this produces a new interval
$[x_1, y_1]$, half the size of the original, satisfying
$f(x_1)f(y_1)<0$. So the process can be repeated, to produce a new
interval $[x_2, y_2]$, and so on. Either this process terminates after
finitely many steps (by finding some $m_n = \frac{1}{2}(x_n+y_n)$ such
that $f(m_n) = 0$) or we obtain two infinite sequences $(x_n)$,
$(y_n)$ defining fluxions $x:= [x_n]$, $y := [y_n]$. Now in the latter
case both $x$ and $y$ are monotonic, and hence convergent by the
monotone convergence theorem.  Furthermore, $y-x \to 0$ is
infinitesimal since $y_n-x_n = \frac{1}{2^n}(b-a)$ for all $n$, so $x$
and $y$ have the same limit $c$. Finally, by construction all the
$f(x_n)$ have the same sign as $f(x_0)=f(a)$, and all the $f(y_n)$
have the same sign as $f(y_0)=f(b)$. It follows that either $f(x)>0$
and $f(y)<0$, or else $f(x)<0$ and $f(y)>0$. In either case, both
$f(x)$ and $f(y)$ converge to $f(c)$ by continuity of $f$, so by
Proposition \ref{le} it follows that $0 \le f(c)\le 0$, and $f(c)=0$.
\end{proof}

Next we consider uniform continuity. In order to formulate it, we
require the following concept.

\begin{defn*}
  Two fluxions $x,y$ are \emph{infinitely close} (written as $x
  \approx y$) if their difference $x-y$ is infinitesimal.
\end{defn*}

It is easily checked that infinite closeness is an equivalence
relation on the set $\dot{\R}$ of fluxions.  Notice that if $a$ is
real then $x \approx a$ if and only if $x \to a$.  It is easy to
construct examples of divergent fluxions $x,y$ with $x \approx y$; for
instance $x=[n]$ and $y=[n+1/n]$.

\begin{defn*}
  Let $f$ be a real-valued function of a real variable, and let $E$ be
  a subset of its domain. We say that $f$ is \emph{uniformly
    continuous} on the set $E$ if $f(x) \approx f(y)$ whenever $x
  \approx y$, for any fluxions $x,y \in \dot{E}$.
\end{defn*}

It is clear that if $f$ is uniformly continuous on $E$ then $f$ is
continuous on $E$ (\emph{i.e.}, continuous at every point of $E$). The
converse fails. For example, the function $f(x) = x^2$ is continuous
on $\R$ but not uniformly continuous on $\R$, because $x=[n]$ and
$y=[n+\frac{1}{n}]$ are infinitely close inputs producing
outputs $$f(x) = [n^2], \quad f(y)=[n^2+2+1/n^2]$$ that are not
infinitely close. 

Let us now consider differentiability. Some slight care is required in
order to ensure that the difference quotient is defined for all
fluxions under consideration.

\begin{defn*}
Let $f\colon D \to \R$ be a real-valued function of a real
variable. Let $a$ be a point of $D$ with the property that there is
some fluxion $x \in \dot{D}$ with $x \to a$ but such that $x$ avoids
$a$. (Such $a$ are called accumulation points of $D$.) Then $f$ is
said to be \emph{differentiable} at $a$ provided that
$\frac{f(x)-f(a)}{x-a}$ converges to a real number $f'(a)$ for every
fluxion $x \in \dot{D}$ with $x$ avoiding $a$ and $x \to a$. The real
number $f'(a)$ is called the derivative of $f$ at $a$.
\end{defn*}

One could derive the standard rules of differentiation at this point.

\begin{prop}
  Differentiability implies continuity.
\end{prop}

\begin{proof}
Put $h = x-a$, where $x \in \dot{D}$ and $x$ avoids $a$. Then $h$ is
an invertible infinitesimal and $\frac{f(x)-f(a)}{h} \to f'(a)$. Since
$h \to 0$ it follows that $h \cdot \frac{f(x)-f(a)}{h} \to 0 \cdot
f'(a)$.  In other words, $f(x) - f(a) \to 0$, so $f(x) \to f(a)$.  If
$x \in \dot{D}$ converges to $a$ but doesn't avoid $a$, then it also
follows that $f(x) \to f(a)$, so the proof is complete.
\end{proof}

Differentials may be treated as usual, by introducing $dx$ as another
independent variable, and if $y = f(x)$ is a function, defining $dy$
by the rule $dy = f'(x)dx$. If $dx$ is allowed to take on
infinitesimal values, then $dy$ becomes infinitesimal as well, for all
values of $x$ for which $f$ is differentiable.

Limits of functions can also be defined using infinitesimal
language. The notion of invertibility in the ring $\dot{\R}$ is a
crucial ingredient of the definition. For the sake of economy, it is
convenient to work within the extended real number system $\R \cup \{
\infty, -\infty\}$. This allows the various cases of infinite limits
as well as finite ones to be treated with a single definition.

\begin{defn*}
Let $f \colon D \to \R$ be a real-valued function of a real variable,
and suppose that there is some fluxion $x \in \dot{D}$ with $x$
avoiding $a$, such that $x \to a$ for some extended real number $a$.
We say that $f(x)$ approaches an extended real number $L$ as $x$
approaches $a$, and write $$\displaystyle \lim_{x\to a} f(x) = L$$ if
$f(x) \to L$ for every $x \to a$ such that $x$ avoids $a$.
\end{defn*}

In other words, the equality $\lim_{x \to a} f(x) = L$ means that
$f(x)$ approaches $L$ for every $x \in \dot{D}$ that approaches and
avoids $a$, and that there must be at least one such $x \in \dot{D}$.
The existence and value of $f(a)$, if any, is irrelevant for the
existence and value of the limit. 

We note also that one-sided limits can be obtained as a special case
of the definition, by merely restricting the domain appropriately. The
details are left to the reader.

At this point one may prove the following standard property of
continuity, sometimes taken as its definition.

\begin{prop}
  Let $f \colon D \to \R$ be a real-valued function of a real
  variable, and suppose $a$ is a point of $D$. Then $f$ is continuous
  at $a$ if and only if $\lim_{x\to a} f(x) = f(a)$.
\end{prop}

\begin{proof}
From the definitions it is clear that if $f$ is continuous at $a$ then
$\lim_{x\to a} f(x) = f(a)$. Conversely, suppose that $\lim_{x\to a}
f(x) = f(a)$. Then $f(x) \to f(a)$ for all $x \to a$ such that $x$
avoids $a$. We need to show that $f(x) \to f(a)$ for all $x \to a$,
even if $x$ does not avoid $a$. So assume that $x$ does not avoid $a$
and $x\to a$. This means that if we write $x = [x_n]$ then for any $N
\in \N$ there are infinitely many values of $n\ge N$ such that
$x_n=a$. It is still true that $f(x) \to f(a)$, so the proof is
complete.
\end{proof}

\section{Topology}\noindent
The language of infinitesimals is appropriate for discussing the
topology of the real line. Indeed, we have already encountered
isolated points and accumulation points. Here are infinitesimal
definitions of the standard topological notions.

\begin{defn*}
Let $E$ be a subset of $\R$. The set $E$ is \emph{closed} if every
convergent fluxion in $\dot{E}$ converges to a point of $E$ (\emph{i.e.},
$x\in \dot{E}$ and $x \to a$ for $a$ real implies $a \in E$). A real
point $a$ is said to be an \emph{accumulation point} of $E$ if there
is some $x \in \dot{E}$ such that $x$ avoids $a$ and $x \to a$.  A
point $a \in E$ is said to be an \emph{isolated} point of $E$ if the
only $x \in \dot{E}$ converging to $a$ is $x = a$.  Finally, we say
that $E$ is \emph{open} if for every $a \in E$, $\dot{E}$ contains
every fluxion $x \in \dot{\R}$ which converges to $a$ (\emph{i.e.}, $x \in
\dot{\R}$ and $x \to a \in E$ implies $x \in \dot{E}$).
\end{defn*}

Clearly, $E$ is closed if and only if $E$ contains all of its limits.
The limit of a fluxion in $\dot{E}$ either belongs to $E$ or is an
accumulation point of $E$, so $E$ is closed if and only if $E$
contains all of its accumulation points. A finite set has no
accumulation points since all its points are isolated, so finite sets
are closed.

The notion of accumulation point already arose in the discussion of
limits of functions. Looking back at the definition, we see that in
order for $\lim_{x\to a} f(x)$ to exist in the extended real number
system, it is necessary that $a$ be an accumulation point for the
domain of $f$.

In order to prove some of the standard topological results about $\R$,
it is useful to have the following lemma available.

\begin{lem*}
  Let $E, F$ be subsets of $\R$, and $\{E_\alpha\}$ a family of
  subsets of $\R$. Then
\item{\upshape(a)} $E \subset F$ implies $\dot{E} \subset \dot{F}$;
\item{\upshape(b)} $\bigcup \big(\dot{E}_\alpha\big)\subset \dot{U}$,
  where $U := \big(\bigcup E_\alpha\big)$;
\item{\upshape(c)} $\bigcap \big(\dot{E}_\alpha\big) = \dot{V}$, where
  $V:= \big(\bigcap E_\alpha\big)$;
\item{\upshape(d)} $E \cap F = \emptyset$ implies $\dot{E} \cap \dot{F}
  = \emptyset$.
\item{\upshape(e)} If $F = \R-E$ is the complement of $E$, then
  $\dot{F} \subset \dot{\R} - \dot{E}$.
\end{lem*}

\begin{proof}
Part (a) is obvious. Part (b) simply says that a fluxion whose range
lies in some $E_\alpha$ lies in the union of all $E_\alpha$, so that
is also clear. 

To get (c), notice that $\big( \bigcap E_\alpha \big) \subset
E_\alpha$ implies by part (a) that $\dot{V} \subset
\dot{E}_\alpha$. Since this holds for any $\alpha$, we have $\dot{V}
\subset \bigcap \big(\dot{E}_\alpha \big)$. For the opposite
inclusion, observe that if $x \in \bigcap \big(\dot{E}_\alpha \big)$
then for any $\alpha$, $x \in \dot{E}_\alpha$, so $x \in \dot{V}$, as
desired.

Part (d) is obvious (it also follows from part (c)), and part (e)
follows immediately from part (d).
\end{proof}

The lemma is used in the proofs of the next two results.

\begin{prop}\label{prop:complements}
  A set $E$ of reals is open if and only if its complement $\R-E$ is
  closed.
\end{prop}

\begin{proof}
Suppose that $E$ is open. Let $F$ be the complement of $E$ and suppose
$x \in \dot{F} = \dot{\R} - \dot{E}$. Assume that $x \to a$ where $a$
is real. If $a \in E$ then $x \in \dot{E}$, a contradiction. So $a \in
F$ and thus $F$ is closed.

On the other hand, if $F =\R-E$ is closed, consider any real $a \in E$
and any fluxion $x$ with $x \to a$.  We must show that $x \in
\dot{E}$.  If not, then any sequence $(x_n)$ defining $x$ contains
infinitely many points of $F$, and thus there would be a subfluxion
$y$ of $x$ with $y \in \dot{F}$.  Since $y \to a$ and $F$ is closed,
this implies that $a \in F$, a contradiction. Thus $x \in \dot{E}$, as
required.
\end{proof}

\begin{prop}
  The union of any collection of open sets in $\R$ is again open. A
  finite intersection of open sets in $\R$ is open.  The empty set and
  $\R$ itself are both open (and closed).  Thus the collection of open
  subsets of $\R$ forms a topology on $\R$.
\end{prop}

\begin{proof}
Let $\{G_\alpha\}$ be an arbitrary collection of open sets, and put $G
= \bigcup_\alpha G_\alpha$. If $a \in G$ then $a \in G_\alpha$ for
some $\alpha$, and thus every fluxion $x$ such that $x \to a$ belongs
to $\dot{G}_\alpha \subset \dot{G}$, so $G$ is open.

Now suppose we have a finite collection $\{G_1, \dots, G_k\}$ of open
sets, and put $H = G_1 \cap G_2 \cap \cdots \cap G_k$. If $a \in H$
then $a \in G_j$ for each $j=1, \dots, k$, so any fluxion $x$ such
that $x \to a$ has the property that $x \in \dot{G}_j$, for $j = 1,
\dots, k$. But then $x \in \dot{G}_1 \cap \dot{G}_2 \cap \cdots \cap
\dot{G}_k = \dot{H}$, and thus $H$ is open in $\R$.
\end{proof}

Compactness may also be defined via infinitesimals. (This is based on
the nontrivial fact that for metric spaces, compactness and sequential
compactness are equivalent.)

\begin{defn*}
Let $K$ be a subset of $\R$.  We say that $K$ is \emph{compact} if
every fluxion $x \in \dot{K}$ has a convergent subfluxion $u \subset
x$, converging to a point of $K$.
\end{defn*}

Clearly, any finite set $K = \{a_1, a_2, \dots, a_k\}$ of real numbers
is compact, because any fluxion $x \in \dot{K}$ must have a constant
subfluxion.

Here are some simple consequences of the definition of compactness.

\begin{prop} \label{thm:cmpct-implies-closed}
Let $K \subset \R$ be compact.
  \item{\upshape(a)} $K$ is closed.
  \item\noindent{\upshape(b)} Any closed subset $F$ of $K$ is compact.
  \item\noindent{\upshape(c)} Any infinite subset $E$ of $K$ must have
    an accumulation point in $K$.
\end{prop}

\begin{proof}
(a) If $x \in \dot{K}$ converges, then any subfluxion of $x$ also
  converges, and converges to the same limit, which must be in $K$ by
  compactness of $K$. So $K$ contains all its limits, and thus is
  closed.

(b) Suppose that $x \in \dot{F}$. Since $\dot{F} \subset \dot{K}$, by
  compactness of $K$ there is a subfluxion $u \subset x$ such that $u
  \to a$ for some $a \in K$. But $u \in \dot{F}$ and $F$ is closed, so
  in fact $a \in F$, as required.

(c) Pick some $x \in \dot{E}$ such that $x$ has no constant
  subfluxion. This is possible because $E$ is infinite. Since $\dot{E}
  \subset \dot{K}$, by compactness of $K$ there must be some
  subfluxion $u \subset x$ such that $u \to a \in K$. Since $x$ has no
  constant subfluxion, $u$ avoids $a$, and thus $a\in K$ is an
  accumulation point of $E$. (Note that $u \in \dot{E}$ since $x \in
  \dot{E}$.)
\end{proof}

It follows immediately from the Bolzano--Weierstrass theorem that any
closed interval $I=[a,b]$ is compact in $\R$. (Because $I$ is closed,
the limit of a convergent subfluxion of any $x \in \dot{I}$ must lie
in $I$.)

\begin{thm*}[Heine--Borel Theorem]
Let $E$ be a subset of $\R$. Then $E$ is compact if and only if $E$ is
closed and bounded.
\end{thm*}

\begin{proof}
Assume $E$ is compact. Then by Proposition
\ref{thm:cmpct-implies-closed}(a), $E$ is closed. To see that $E$ is
bounded, argue by contradiction. Assuming the contrary, we can find a
fluxion $x=[x_n]$ in $\dot{E}$ such that $|x_n| > n$ for each natural
number $n$. The fluxion $x$ is not finite. No subfluxion of $x$ can
converge to a point of $E$, since convergent fluxions are finite. This
violates the compactness of $E$.

For the converse, assume that $E$ is closed and bounded. Since $E$ is
bounded, we can find some interval $I = [-r,r]$ containing $E$.  Since
$I$ is compact, by Proposition \ref{thm:cmpct-implies-closed}(b) it
follows that $E$ is compact.
\end{proof}

As an application of the last result, we can prove that every
continuous real-valued function of a real variable achieves a maximum
and minimum on any closed interval. For this we first need to observe
that \emph{the continuous image of a compact set must be compact}. For
if $f$ is continuous and $E$ a compact subset of its domain, then put
$F:= f(E)$. If $y = f(x) \in \dot{F}$ for some $x \in \dot{E}$, then
by compactness of $E$ there is a subfluxion $u \subset x$ with $u \to
a$ where $a$ is real and $a \in E$. By continuity, $f(u) \to f(a) \in
F$, and clearly $f(u) \subset f(x)$.

\begin{thm*}[Extreme Values Theorem]
  A continuous real-valued function $f \colon D \to \R$ of a real
  variable attains a maximum and minimum on any closed interval $I =
  [a,b]$ contained in its domain. 
\end{thm*}

\begin{proof}
Put $K = f(I)$. Since $K$ is compact, it is closed and bounded by the
Heine--Borel theorem. Boundedness implies the existence of the real
numbers $m:= \inf K$ and $M:= \sup K$. We need to show that both $m$
and $M$ belong to $K$. But surely there is a sequence of points in $K$
approaching $m$. By compactness, this forces $m \in K$. Similarly for
$M$.
\end{proof}

From this one can derive the mean value theorem and its various
consequences (\emph{e.g.}, Taylor's theorem). 

One could go on to treat the topology of metric spaces similarly, in
particular extending the results of this section to subsets of $\R^n$
and $\C^n$.  If $X$ is a metric space with metric $d$, one would
introduce the set $\dot{X}$, consisting of fluxions taking values in
$X$, which would be defined as equivalence classes of sequences in
$X^\N$ under eventual equality. The object $\dot{X}$ thus produced is
not in general a ring, but that is of no consequence.  The main point
is that the metric $d$ on $X$ extends to elements of $\dot{X}$, and
thus one could define, for example, convergence of a fluxion $x \in
\dot{X}$ to a point $a \in X$ to occur if and only if $d(x,a)$ is an
infinitesimal real fluxion.  The definitions of topological notions in
$X$ would then take precisely the same form as those given in this
section.  We leave the details to the interested reader.

\section{Postscript}\noindent
Let us justify the description of fluxions as germs of real-valued
sequences.

Regard the set $\N$ of natural numbers as a topological space with the
cofinite topology. In this topology, the open sets are precisely the
complements of finite sets. Let $\widehat{\N} = \N \cup \{\infty\}$ be
the one-point compactification of $\N$. The neighborhoods of $\infty$
in $\widehat{\N}$ are the subsets of the form $E \cup \{\infty\}$,
where $E$ is the complement of a finite subset of $\N$.

Now recall a standard general construction in topology.  Given a point
$x$ of a topological space $X$, and two maps $f, g \colon X \to Y$
(where $Y$ is any set), then $f$ and $g$ define the same germ at $x$
if there is a neighborhood $U$ of $x$ such that the restrictions of
$f$ and $g$ to $U$ coincide. Defining the same germ at $x$ gives an
equivalence relation on the space $Y^X$ of functions from $X$ into
$Y$. The resulting equivalence classes are called \emph{germs}.

Thus we see that the ring $\dot{\R}$ may be identified with the set of
germs of real-valued sequences (extended to $\widehat{\N}$) at
infinity.

\end{document}